\documentclass[11pt]{article}

\usepackage{amssymb,amsmath,amsthm,rotating,setspace}

\numberwithin{equation}{section}

\newcommand{\n}{\noindent}

\theoremstyle{plain}
\newtheorem{thm}{Theorem}[section]
\newtheorem{lem}[thm]{Lemma}

\newtheorem{pro}[thm]{Proposition}
\newtheorem{cor}[thm]{Corollary}
\newtheorem{que}[thm]{Question}

\theoremstyle{definition}
\newtheorem{defn}[thm]{Definition}

\newtheorem{rem}[thm]{Remark}

\begin{document}

\title{Connections of the Corona Problem with Operator Theory and Complex Geometry}

\author{Ronald G.~Douglas}

\date{}
\maketitle

\abstract{
The corona problem was motivated by the question of the density of the open unit disk $\mathbb{D}$ in the maximal ideal space of the algebra, $H^\infty (\mathbb{D})$,  of bounded holomorphic 
functions on $\mathbb{D}$.  In this note we study relationships of the problem with questions in operator theory and complex geometry.  We use the framework of Hilbert modules focusing on 
reproducing kernel Hilbert spaces of holomorphic functions on a domain, $\Omega$, in $\mathbb{C}^m$.  We interpret several of the approaches to the corona problem from this point of view.  A 
few new observations are made along the way.}

\vfill

\n 2012 MSC: 46515, 32A36, 32A70, 30H80, 30H10, 32A65, 32A35, 32A38

\n Keywords: corona problem, Hilbert modules, reproducing kernel Hilbert space, commutant lifting theorem

\section{Introduction}

\indent

Frequently, questions in abstract functional analysis lead to very concrete problems in ``hard analysis'' with results on the latter having stronger implications in the larger abstract context.
That is the case with the corona problem and describing some of these larger relationships is the main goal of this note.  In this problem in analysis, however, the mix is quite broad and 
includes aspects of complex geometry and function theory as well as operator theory.  

Almost all the results mentioned are due to others, as we will try to make clear, and the novelty, if any,  is in the organization and point of view.  There are a few observations in Sections
6 and 7 which may be new.  In particular, we show if the multiplier algebra for a reproducing kernel Hilbert module is logmodular, then the Toeplitz corona theorem implies the corona theorem.  
Further, for all reproducing kernel Hilbert modules, knowing that the Toeplitz corona theorem holds for all cyclic submodules, with a \emph{uniform} lower bound implies the corona theorem.

The corona problem springs from the study of uniform algebras and, in particular, the description of their maximal ideal spaces.  Following Gelfand's fundamental results on commutative Banach
algebras, calculating the space of maximal ideals for particular commutative Banach algebras became a central issue in functional analysis.  In 1941 Kakutani raised the question for the
maximal ideal space for the algebra, $H^\infty (\mathbb{D})$, of bounded holomorphic functions on the open unit disk $\mathbb{D}$ and whether or not the open disk $\mathbb{D}$ is dense in it.  
One quickly reduces the latter problem to showing for $\{\varphi_i\}^{n}_{i=1}$ in $H^\infty (\mathbb{D})$, satisfying $$\sum^{n}_{i=1} |\varphi_i (z)|^2 \geq \epsilon^2 >0$$ for 
$z \in \mathbb{D}$, that there exists an $n$-tuple $\{\psi_i \}^{n}_{i=1} \subset H^\infty (\mathbb{D})$ such that $$\sum^{n}_{i=1} \varphi_i (z) \psi_i (z)=1$$ for $z \in \mathbb{D}$.  

If $\{\varphi_i\}^{n}_{i=1}$ consists of continuous functions on the closure $\bar{\mathbb{D}}$ of $\mathbb{D}$ or are in $C(\bar{\mathbb{D}})$, then it is not hard to show that an $n$-tuple 
$\{\psi_i\}^{n}_{i=1}$ exists in $C(\bar{\mathbb{D}})$ satisfying $$\sum^{n}_{i=1} \varphi_i \psi_i =1 \text{ in }C(\hat{\mathbb{D}}).$$  If one could control the supremum norms 
$\{||\psi_i||_\infty \}^{n}_{i=1}$ in terms of the norms $\{||\varphi_i||_\infty\}^{n}_{i=1}$ and $\epsilon$, then a standard normal families argument from function theory would establish the 
result for $H^\infty (\mathbb{D})$.  But the easy proof for $C(\bar{\mathbb{D}})$ doesn't yield bounds.  (A slightly more general result for holomorphic functions continous on the closure of a 
domain, with proof, appears in Section 7.)  However, even if one tries to extend this result for $C(\bar{\mathbb{D}})$ to $H^\infty (\mathbb{D})$ in some other way, one quickly sees that in 
most applications of the  solution to the corona problem, one needs to discuss such bounds anyway.  

In this note we will cite the relevant sources for the results discussed but refer the reader to [2] for more detailed references and historical remarks.  

\section{The Corona Problem}
Although the corona problem could be formulated more generally, here we restrict attention to algebras of bounded holomorphic functions.  Let $\Omega$ be a bounded connected domain in 
$\mathbb{C}^m$ for some positive integer $m$ and let $H^\infty (\Omega)$ denote the algebra of bounded holomorphic functions on $\Omega$.  Define the supremum norm by 
$$||\varphi||_\infty = \sup_{\omega \in \Omega} |\varphi(\omega)| \text{ for }\varphi \in H^\infty (\Omega).$$  Then $H^\infty (\Omega)$ is a commutative Banach algebra for the pointwise 
algebraic operations and hence there is a compact Hausdorff space, $M_{H^\infty (\Omega)}$, of maximal ideals.  Recall, following Gelfand, that a maximal ideal $I$ in the commutative Banach 
algebra $H^\infty (\Omega)$ can be identified with a multiplicative linear functional on $H^\infty (\Omega)$.  This follows since $H^\infty (\Omega)/I \cong \mathbb{C}$ as algebras by Gelfand's 
theorem.  The family of all such maps lies in the unit ball of the Banach space dual of $H^\infty (\Omega)$ and, in the weak*-topology, is a compact Hausdorff space, $M_{H^\infty (\Omega)}$, 
called the maximal ideal space.  

For $\omega \in \Omega$, the set $I_\omega =\{\varphi \in H^\infty (\Omega): \varphi(\omega) = 0 \}$ is a maximal ideal in $H^\infty (\Omega)$ with the corresponding multiplicative linear 
functional defined as evaluation at $\omega$.  Thus one has an imbedding, $\Omega \subseteq M_{H^\infty (\Omega)}$, of $\Omega$ in $M_{H^\infty (\Omega)}$.  The corona problem asks: Is $\Omega$
dense in $M_{H^\infty (\Omega)}$?  Let $\text{clos }\Omega$ denote the closure of $\Omega$ in $M_{H^\infty (\Omega)}$.  The complement, $M_{H^\infty (\Omega)} \setminus \text{clos }\Omega$, if 
non empty, is said to be the \emph{corona} for $H^\infty (\Omega)$.  However, there is a small problem.

There exist connected bounded domains $\Omega_1 \subseteq \Omega_2 \subseteq \mathbb{C}^m$ so that $H^\infty (\Omega_1)=H^\infty (\Omega_2)$ in the sense that every 
$\varphi \in H^\infty (\Omega_1)$ extends to a bounded holomorphic function on $\Omega_2$.  In such cases, not all points in $M_{H^\infty (\Omega_1)} \setminus \text{clos }\Omega_1$ should be 
considered to be in the corona; in particular, the points in $\Omega_2 \setminus \bar{\Omega}_1$ should not be in the corona.  Hence we should, and do, restrict attention to bounded domains 
$\Omega$ for which no such larger domain exists.  Hence, we make the following assumption from now on:\\
\emph{$\Omega$ is a bounded connected domain in $\mathbb{C}^m$ for which no super domain exists supporting the same algebra of bounded holomorphic functions.}

Domains such as the unit ball, $\mathbb{B}^m$, or the polydisk, $\mathbb{D}^m$, have this property but, for example, the domain between two spheres in $\mathbb{C}^m$ does not.  For an 
$\Omega \subseteq \mathbb{C}^m$, consider the map $\pi_\Omega :M_{H^\infty (\Omega)} \to \mathbb{C}^m$ such that $\pi_\Omega (x)=(\hat{z}_1 (x),...,\hat{z}_m (x))$, where 
$\{ \hat{z}_i \}^{m}_{i=1}$ are the Gelfand transforms of the functions $\{ z_i \}^{m}_{i=1} \subseteq H^\infty (\Omega)$.  Then the open subset of $M_{H^\infty (\Omega)}$ on which $\pi_\Omega$
is locally one-to-one defines the largest domain in $\mathbb{C}^m$ containing $\Omega$ supporting the same algebra of bounded holomorphic functions.  Since in this note this topic is only 
peripheral for us, we won't consider the problem further of characterizing the property.

We now return to the corona problem.

\begin{itemize}
 \item[ ] If $\{\varphi_i\}^{n}_{i=1} \subseteq H^\infty (\Omega)$ satisfies 
 \item[(0)] $$\sum^{n}_{i=1} |\varphi_i (\omega)|^2 \geq \epsilon^2 >0$$ for $\omega \in \Omega$ and $\Omega$ is dense in 
 $M_{H^\infty (\Omega)}$, then inequality (0) extends to all of $M_{H^\infty (\Omega)}$.  Therefore, the ideal 
 $$J = \left\{ \sum^{n}_{i=1} \varphi_i \psi_i : \{ \psi_i \}^{n}_{i=1} \subseteq H^\infty (\Omega) \right\},$$ generated by the set $\{\varphi_i \}^{n}_{i=1}$, is not contained in any proper 
 maximal ideal.  Hence the function $1\in J$, or
 \item[(1)] There exists $\{\psi_i\}^{n}_{i=1} \subseteq H^\infty (\Omega)$ such that $$\sum^{n}_{i=1} \varphi_i (\omega) \psi_i (\omega)=1$$ for $\omega \in \Omega$ or 
 $$\sum^{n}_{i=1} \varphi_i \psi_i=1\text{ in }H^\infty (\Omega).$$
\end{itemize}

\noindent Thus, we can restate the corona problem:\\
For $\{\varphi_i\}^{n}_{i=1} \subseteq H^\infty (\Omega)$ does (0) imply (1)?  

It is clear that (1) implies (0).  In fact, a simple application of the Cauchy-Schwarz inequality yields
\begin{equation*}
 1=|\sum^{n}_{i=1} \varphi (\omega) \psi_i (\omega)| \leq \left( \sum^{n}_{i=1} |\varphi_i (\omega)|^2 \right)^{\frac{1}{2}} \left( \sum^{n}_{i=1} |\psi_i (\omega)|^2 \right)^{\frac{1}{2}}
 \leq \left\{ \sum^{n}_{i=1} |\varphi_i (\omega)|^2 \right\}^{\frac{1}{2}} \left\{ \sum^{n}_{i=1} ||\psi_i ||^{2}_{\infty} \right\}^{\frac{1}{2}}
\end{equation*}

\noindent or 
\begin{equation*}
 \sum^{n}_{i=1} |\varphi_i (\omega)|^2 \geq 1/ \left\{ \sum^{n}_{i=1} ||\psi_i||^{2}_{\omega} \right\}.
\end{equation*}

In 1962 Carleson settled the corona problem for $H^\infty (\mathbb{D})$ in the affirmative using techniques from harmonic analysis and function theory [1].  Since then, many results have been 
obtained, in both the affirmative and negative, for the corona problem for a variety of domains in $\mathbb{C}^m$ (cf. [2]).  We will not pursue these results or their proofs in this note.  
Rather we want to explore various connections and relationships of the corona problem with operator theory and complex geometry.

\section{Hilbert Modules and the Corona Problem}
Although there is no Hilbert space mentioned in the statement of the corona problem, there is a natural way to relate the corona problem for function algebras to operator theory.  Recall that a
Hilbert space $\mathcal{R}$ of holomorphic functions on a bounded, connected domain $\Omega$ of $\mathbb{C}^m$ is said to be a \emph{reproducing kernel Hilbert space (RKHS)} if 
$\mathcal{R} \subseteq \emph{O}(\Omega)$ such that the evaluation map $ev_\omega (f) = f(\omega)$ for $f \in \mathcal{R}$ is bounded for $\omega \in \Omega$, where $\emph{O} (\Omega)$ denotes 
the space of holomorphic functions on $\Omega$.  

One obtains the usual ``two-variable'' kernel function, $K(z, \omega)$, on $\Omega \times \Omega$ by setting 
$K(z, \omega) = ev_z ev^{*}_{\omega} \in \mathcal{L}(\mathbb{C},\mathbb{C}) \cong \mathbb{C}$ for $z,\omega \in \Omega$.  The unique function $k_\omega$ satisfying 
$ev_\omega f=< f, k_\omega >_{\mathcal{R}}$ is obtained by setting $k_\omega (z) = K(z,\omega)$.

If $z_i \mathcal{R} \subseteq \mathcal{R}$ for $i=1, 2, ..., m$, then $\mathcal{R}$ is a Hilbert module over the algebra of polynomials in $m$ variables , $\mathbb{C}[z_1, ..., z_m]$.  In this 
case we say that $\mathcal{R}$ is a \emph{reproducing kernel Hilbert module (RKHM)}.

Although it is not necessary, to simplify matters we assume that $1 \in \mathcal{R}$ which implies $\mathbb{C}[z_1, ..., z_m ]\subseteq \mathcal{R}$.  In general, $\mathbb{C}[z_1, ..., z_m]$ is
not dense in $\mathcal{R}$ but it is in many natural examples, such as for $\mathcal{R}$ the Hardy space, $H^2 (\partial \mathbb{B}^m)$, on the unit ball, $\mathbb{B}^m$, in $\mathbb{C}^m$ or 
the Bergman space $L^{2}_{a} (\Omega)$, where the former space can be defined as the closure of $\mathbb{C}[z_1, ..., z_m]$ in $L^2 (\partial \mathbb{B}^m)$ for Lebesgue measure on the unit 
sphere, $\partial \mathbb{B}^m$, and the latter space is the closure of $\mathbb{C}[z_1,...,z_m]$ in $L^2 (\Omega)$ for Lebesgue measure on $\Omega$.  (Actually, one might want to put 
some restrictions on $\Omega$ and, perhaps, let $L^{2}_{a} (\Omega)$ be the closure of $H^\infty (\Omega)$, but we do not go into any detail here.)

For $\mathcal{R}$ a RKHM, a function $\psi \in \emph{O}(\Omega)$ is said to be a \emph{multiplier} for $\mathcal{R}$ if $\psi\mathcal{R} \subseteq \mathcal{R}$.  The set of multipiers,
$\mathcal{M}(\mathcal{R})$, forms a commutative Banach algebra which one can show is contained in $H^{\infty}(\Omega)$.
\bigskip

We can formulate a corona problem in the context of a $\mathcal{R}KHM \thickspace \mathcal{R}$ as follows:

\begin{itemize}
 \item The $n$-tuple $\{ \varphi_i\}^{n}_{i=1} \subseteq \mathcal{M}(\mathcal{R})$ satisfies (0) or
  $$\sum^{n}_{\i=1} |\varphi_i (\omega)|^2 \geq \epsilon^2 >0\text{, for }\omega \in \Omega.$$

 \item Statement (1) holds or there exists $\{ \psi_i \}^{n}_{i =1} \subseteq \mathcal{M}(\mathcal{R})$ such that $$\sum^{n}_{i=1} \varphi_i (\omega) \psi_i (\omega) =1$$ for 
 $\omega \in \Omega$.
\end{itemize}
\noindent The corona problem for $\mathcal{M}(\mathcal{R})$ asks if (0) implies (1).  As before, (1) implies (0).

An immediate connection with operator theory concerns the Hilbert module sequence
\begin{itemize} 
\item [(2)] $ 0 \to \mathcal{R} \xrightarrow{M_\Phi} \mathcal{R} \otimes \mathbb{C}^n \xrightarrow{\pi_\Phi} \mathcal{R}_\Phi \to 0,$
\end{itemize}

\noindent which is exact if range $M_\Phi$ is closed, where $\mathcal{R}_\Phi$ is the quotient Hilbert module $\mathcal{R} / \text{ range } M_\Phi$ and $\pi_\Phi$ is the quotient map.  If (1) holds, 
then range $M_\Phi$ is closed.  Here, $M_\Phi :\mathcal{R}\to \mathcal{R}\otimes \mathbb{C}^n$ is defined 
$$\mathcal{M}_\Phi f=\sum^{n}_{i=1} \varphi_i f\otimes e_i,$$ where $\{e_i \}^{n}_{i=1}$ is the standard orthonormal basis for $\mathbb{C}^n$.

We will call a RKHM $\mathcal{R}$ \emph{subnormal} if there exists a probability measure $\mu$ on $\bar{\Omega}$, the closure of $\Omega$ in $\mathbb{C}^m$, such that 
$\mathcal{R} \subseteq L^2 (\mu)$ isometrically.  Although many RKHM are subnormal, not all are.  Perhaps the simplist example of this phenomenon is the Dirichlet space on $\mathbb{D}$.  
Another family of recent interest is the Drury-Arveson space $H^{2}_{m}$ on $\mathbb{B}^m$ with kernel function $(1-\left< z,\omega \right>_{\mathbb{C}^m})^{-1}$.  An alternate description of 
$H^{2}_{m}$ is that it is the symmetric Fock space.

The first question concerning the exactness of (2) is the exactness at $\mathcal{R}$, the left-most module.  If $\Omega$ is connected and $\Phi \neq 0$, then $M_\Phi$ is one-to-one.  Hence, 
exactness comes down to the closure of the range of $M_\Phi$.  This issue is transparent in the subnormal case.

\begin{pro}
 For $\{ \varphi_i\}^{n}_{i=1} \subseteq \mathcal{M}(\mathcal{R})$ with $\mathcal{R}$ a subnormal RKHM on $\Omega$, (0) implies that range $M_\Phi$ is closed or that the sequence $(2)$ 
 is exact at $\mathcal{R}$ and hence (2) is a short exact sequence.
\end{pro}

\begin{proof}
 For $f\in \mathcal{R}$ we have
 \begin{equation*}
  ||M_\Phi f||^2 = \sum^{n}_{i=1}||\varphi_i f||^2 = \int_{\bar{\Omega}} \sum^{n}_{i=1} |\varphi_i f|^2 d\mu = \int_{\bar{\Omega}} \left( \sum^{n}_{i=1} |\varphi_i |^2 \right) 
  |f|^2 d\mu \geq \epsilon^2 ||f||^2,
 \end{equation*}
\noindent where $\mathcal{R} \subseteq L^2 (\mu)$ for the probability measure $\mu$ on $\bar{\Omega}$.  Thus $M_\Phi$ is bounded below and hence has closed range.
\end{proof}

\begin{rem}\label{3.2}
 Note that the assumption that range $M_\Phi$ is closed implies that $M^{*}_{\Phi}$ is onto and vice-versa.  
\end{rem}

The relation between the module sequence $(2)$ and the corona problem is strong.

\begin{rem}\label{3.3}
 Suppose $\{ \psi_i \}^{n}_{i=1} \subseteq \mathcal{M}(\mathcal{R})$ satisfies $$\sum^{n}_{i=1} \varphi_i \psi_i =1.$$  
\end{rem}

\noindent Then, as mentioned above, this implies (0).  If one defines the operator $N_\Psi:\mathcal{R}\otimes \mathbb{C}^n \to \mathcal{R}$ such that 
$$N_\Psi (\sum^{n}_{i=1} f_i \otimes e_i) = \sum^{n}_{i=1} \psi_i f_i \text{ for }\sum^{n}_{i=1} f_i \otimes e_i \in \mathcal{R}\otimes \mathbb{C}^n ,$$ then $N_\Psi M_\Phi=I_{\mathcal{R}}$ and
hence range $M_\Psi$ is closed.  Thus the failure of range $M_\Phi$ to be closed is an obstruction to an affirmative answer to the corona problem on $\mathcal{M} (\mathcal{R})$.  

We can ask further whether the closeness of range $M_\Phi$ implies (1).  The question is especially relevant for the non-subnormal case.  (We will have more to say about this matter in Section
8.)

To investigate questions such as this one, it is useful to know what module maps look like between modules of the form $\mathcal{R} \otimes \mathbb{C}^k$ and 
$\mathcal{R} \otimes \mathbb{C}^l$ for positive integers $k$ and $l$.

\begin{lem}\label{3.4}
 For positive integers $k,l$, a module map $X:\mathcal{R} \otimes \mathbb{C}^k \to \mathcal{R} \otimes \mathbb{C}^l$ (that is, one that satisfies 
 $(M_\psi \otimes I_{\mathbb{C}^l})X=X(M_\psi \otimes I_{\mathbb{C}^k})$ for $\psi \in \mathcal{M}(\mathcal{R})$) iff there exists 
 $\{ \varphi_{ij} \}^{l}_{j=1} {}^{k}_{i=1} \subseteq \mathcal{M}(\mathcal{R})$ such that $$X \left( \sum^{k}_{i=1} f_i \otimes e_i \right)=\sum^{l}_{j=1} \sum^{k}_{i=1} M_{\varphi_{ij}} 
 f_i \otimes e_j$$ for $$\sum^{k}_{i=1} f_i \otimes e_i \in \mathcal{R} \otimes \mathbb{C}^k.$$
\end{lem}

\begin{proof}
 Standard linear algebra calculations yield the $\{ \varphi_{ij} \}^{k}_{i=1} {}^{l}_{j=1} \subseteq \mathcal{M}(\mathcal{R}).$
\end{proof}

Although this result was doubtless known to many, one can find it in [6].

\begin{pro}\label{3.5}
 Let $\{ \varphi_i \}^{n}_{i=1} \subseteq \mathcal{M}(\mathcal{R})$ for the RKHM $\mathcal{R}$ over $\Omega \subseteq \mathbb{C}^m$ such that $M_\Phi$ has closed range and hence (2) is exact.
 Then the following statements are equivalent to (1):
 \begin{itemize}
  \item[(3)] $M_\Phi$ has a left module inverse $N_\Psi : \mathcal{R} \otimes \mathbb{C}^n \to \mathcal{R}$. 
  \item[(4)] There exists a right module inverse $\sigma_\Phi$ for $\pi_\Phi$.
  \item[(5)] There exists a module idempotent $E$ on $\mathcal{R} \otimes \mathbb{C}^n$ with range $E= \text{ range }M_\Phi$.
 \end{itemize}
\end{pro}

\begin{proof}
 For a short exact sequence of modules in the algebra category, (3), (4) and (5) are always equivalent and this fact carries over to our context.  But let us provide a complete proof since the
 techniques are relevant to later issues.

If $X$ is a left module inverse for $M_\Phi$, then by Lemma 3.4 there exists $\{ \psi_i \}^{n}_{i=1} \subseteq \mathcal{M}(\mathcal{R})$ such that $X=N_\Psi$.  Moreover, $XM_\Phi=I_\mathcal{R}$
implies that $$\sum^{n}_{i=1} \varphi_i \psi_i =1.$$  Hence, $N_\Psi M_\Phi =I_\mathcal{R}$ and we see that (1) and (3) are equivalent.

If $X$ is a left module inverse for $M_\Phi$, then $E=M_\Phi X$ is a module idempotent on $\mathcal{R} \otimes \mathbb{C}^n$ with range $E=\text{ range } M_\Phi$.  Therefore, (3) implies (5)
and one can define the right module inverse $Y=E\pi^{-1}_{\Phi} : \mathcal{R}_\Phi \to \mathcal{R} \otimes \mathbb{C}^n$.  Conversely, if $Y$ is a right module inverse for $\pi_\Phi$, then 
$E' =Y\pi_\Phi$ is a module idempotent on $\mathcal{R} \otimes \mathbb{C}^n$ and we can define the left module inverse 
$X=M^{-1}_{\Phi} (I_{\mathcal{R}\otimes \mathbb{C}^n} -E'): \mathcal{R}\otimes \mathbb{C}^n \to \mathcal{R}$ for $M_\Phi$ since range 
$(I_{\mathcal{R}\otimes \mathbb{C}^n} -E') = \text{ range } M_\Phi$.  But the existence of $E'$ is enough to define $X$ or (5) implies (3) and this completes the proof.
\end{proof}

\begin{rem}\label{3.6}
 Note that in these constructions, $E+E' = I_{\mathcal{R}\otimes \mathbb{C}^n}$.
\end{rem}

\begin{rem}\label{3.7}
 Note that since $\pi_\Psi$ is onto, it always has a right inverse.  However, a right inverse, in general, is not a module map.
\end{rem}

\section{Module Idempotents and Sub-Bundles}

The previous result makes clear the importance of constructing or identifying module idempotents on $\mathcal{R} \otimes \mathbb{C}^n$ for $\mathcal{R}$ a RKHM over a domain 
$\Omega \subseteq \mathbb{C}^m$ with range equal to range $M_\Phi$ for $\Phi$ given by $\{ \varphi_i \}^{n}_{i=1} \subseteq \mathcal{M}(\mathcal{R})$ assuming range $M_\Phi$ is closed.  If $E$ 
is such a module idempotent on $\mathcal{R} \otimes \mathbb{C}^n$, then by Lemma 3.4 there exists $\{ \sigma_{ij} \}^{n}_{i=1} {}^{n}_{j=1} \subseteq \mathcal{M}(\mathcal{R})$ such that 
$$E \left( \sum^{n}_{i=1} f_i \otimes e_i \right) = \sum^{n}_{i=1} \sum^{n}_{j=1} M_{\sigma_{ij}} f_i \otimes e_j \text{ for } \sum^{n}_{i=1} f_i \otimes e_i \in \mathcal{R} \otimes 
\mathbb{C}^n.$$
If for $\omega \in \Omega$ we set $E(\omega) = \sigma_{ij} (\omega)^{n}_{i=1} {}^{n}_{j=1} \in \mathcal{L}(\mathbb{C}^n)$, then the fact that $E^2 =E$ implies $E(\omega)^2 =E(\omega)$ and vice 
versa.

Since the $\sigma_{ij}$ are holomorphic functions and multipliers for $\mathcal{R}$, it follows that $E(\omega)$ is holomorphic (and more) on $\Omega$.  The range of $E(\omega)$ defines a holomorphic 
sub-bundle $\hat{E}$ of $\Omega \times \mathbb{C}^n$ with fiber equal to range $E(\omega)$ at $\omega \in \Omega$.  Moreover, $E^{'}=I-E$ is the complementary idempotent and it defines a 
complementary sub-bundle $\hat{E}^{'}$ of $\Omega \times \mathbb{C}^n$.  That is, $\Omega \times \mathbb{C}^n = \hat{E}^{'} \dotplus \hat{E}$, where ``$\dotplus$'' denotes a (skew) linear 
direct sum of vector bundles.  In particular, $\hat{E}$ and $\hat{E}^{'}$ are not necessarily orthogonal in $\Omega \times \mathbb{C}^n$.

The converse is also true but we must be careful to state it correctly, which involves a mixture of operator theory and complex geometry.

\begin{pro}\label{4.1}
 For $\mathcal{R}$ a RKHM over $\Omega \subset \mathbb{C}^m$ and a submodule $\mathcal{M}$ of $\mathcal{R} \otimes \mathbb{C}^n$,
 the following are equivalent:
 \begin{itemize}
  \item [(i)] There exists a submodule $\mathcal{N}$ of $\mathcal{R} \otimes \mathbb{C}^n$ such that $\mathcal{R} \otimes \mathbb{C}^n = \mathcal{M} \dotplus \mathcal{N}$.
  \item [(ii)] There exists a module idempotent $E$ on $\mathcal{R} \otimes \mathbb{C}^n$ such that $\mathcal{M}=\text{range } E$.
 \end{itemize}
\end{pro}

\begin{proof}
 The proof follows basically from algebra and by invoking the closed graph theorem to conclude the boundedness of $E$.  In particular, if (ii) holds, then setting 
 $\mathcal{N} = (I_{\mathcal{R}\otimes \mathbb{C}^n}-E) \mathcal{R} \otimes \mathbb{C}^n$ yields the complementary submodule $\mathcal{N}$.  Conversely, if (i) holds, then one can define an idempotent $E$ with range 
 $\mathcal{M}$ by setting $Ex=y_1$, where $x=y_1 + y_2$ is the unique decomposition of $x$ with $y_1 \in \mathcal{M}$ and $y_2 \in \mathcal{N}$.
\end{proof}

\begin{rem}\label{4.2}
 If $\mathcal{R} \otimes \mathbb{C}^n = \mathcal{M} \dotplus \mathcal{N}$, then one can define sub-bundles $E$ and $F$ of $\Omega \times \mathbb{C}^n$ such that 
 $\Omega \times \mathbb{C}^n = E \dotplus F$ and, most important,
 $$ \mathcal{M}= \{ f \in \mathcal{R} \otimes \mathbb{C}^n : f(\omega) \in E(\omega), \quad \omega \in \Omega \},$$
 where $E(\omega)$ is the fiber of $E$ at $\omega \in \Omega$.  The same is true for $F$ and $\mathcal{N}$.  However, there is no simple converse to this relationship.  In particular, if one
 expresses $\Omega \times \mathbb{C}^n = E \dotplus F$, where $E$ and $F$ are holomorphic sub-bundles, it need not be the case that
 \begin{itemize}
  \item [(iii)] $\mathcal{R} \otimes \mathbb{C}^n = \{ f\in \mathcal{R} \otimes \mathbb{C}^n : f(\omega) \in E(\omega),\quad \omega \in \Omega \} \dotplus \{ f\in \mathcal{R} \otimes \mathbb{C}^n : 
  f(\omega) \in F(\omega),\quad \omega \in \Omega \}.$
  \end{itemize}
 If we add this assumption for sub-bundles $E$ and $F$ of $\Omega \times \mathbb{C}^n$, then (iii) is equivalent to (i) and (ii) of Proposition 4.1
\end{rem}

These relationships are at the heart of Nikolski's lemma as presented by Treil and Wick [10] in their approach to the corona theorem as follows.

We begin by placing their framework in the context of the short exact sequence $(2)$.  

Assume that $\{ \varphi_i \}^{n}_{i=1} \subseteq \mathcal{M}(\mathcal{R})$ for some RKHM $\mathcal{R}$ over $\Omega \subseteq \mathbb{C}^m$, satisfies (0).  For $\omega \in \Omega$ let 
$P(\omega)$ denote the orthogonal projection of $\mathbb{C}^n$ onto range $\Phi(\omega) \subset \mathbb{C}^n$.  Since (0) holds, rank $\Phi(\omega) =1$ for $\omega \in \Omega$ which implies 
that $\amalg_{\omega \in \Omega}$ range $P(\omega)$ defines a Hermitian holomorphic line bundle over $\Omega$.  \emph{However}, the function $P(\omega)$ is holomorphic only when it is constant.
But there are other idempotent-valued functions $E(\omega)$ on $\Omega$ with range $E(\omega)=\text{ range }P(\omega) = \text{ range }\Phi(\omega)$ for $\omega \in \Omega$ and one of them 
might be both bounded and holomorphic on $\Omega$ (and in $\mathcal{M}(\mathcal{R})$ which is what is required if $\mathcal{M}(\mathcal{R}) \neq H^\infty (\Omega)$ to define the module 
idempotent, denoted by $E$, needed in the proof of Proposition 3.5.  We capture this fact in the following statement.

\begin{pro}\label{4.3}
 For $\{ \varphi_i \}^{n}_{i=1} \subseteq \mathcal{M}(\mathcal{R})$ for the RKHM $\mathcal{R}$ over $\Omega \subseteq \mathbb{C}^n$ with $\mathcal{M}(\mathcal{R})=H^{\infty} (\Omega)$ 
 satisfying (0), statement (1) is equivalent to
 \begin{itemize}
  \item[(6)] There exists a bounded, real-analytic function $V(\omega): \Omega \to \mathcal{L}(\mathbb{C}^n)$ such that
  \begin{align*}
   \text{(a) } \quad &E(\omega) = P(\omega) + V(\omega) \text{ is bounded and holomorphic on } \Omega \text{ and}\\
   \text{(b) } \quad &V(\omega) = P(\omega) V(\omega) (I_{\mathbb{C}^n}-P(\omega)) \text{ for } \omega \in \Omega.
  \end{align*}
  \end{itemize}
\end{pro}

\begin{proof}
 If (1) holds, then by Proposition 3.5 there exists a module idempotent $E$ on $\mathcal{R}\otimes \mathbb{C}^n$ such that range $P(\omega) = \text{ range} M_\Phi=\text{ range }E(\omega)$.  But 
 the latter implies that range $P(\omega)=\text{ range} E(\omega)$ or that (a) holds.  Since $E(\omega)$ is an idempotent, it follows that $V(\omega)=E(\omega)-P(\omega)$ satisfies (b) which 
 concludes the proof that (1) implies (6).

Conversely, if such a function $V(\omega)$ exists, then by setting $E(\omega)=P(\omega)+V(\omega)$ we obtain a bounded holomorphic idempotent map on $\Omega$ such that range 
$E(\omega)= \text{ range } \Phi(\omega)$ for $\omega \in \Omega$.  Therefore, since $\mathcal{M}(\mathcal{R}) = H^\infty (\Omega)$, we can define a module idempotent $E=M_E$ such that range 
$E= \text{ range }M_\Phi$.  Returning to Proposition 3.5, we see that (4) holds which implies (1).
\end{proof}

\begin{rem}\label{4.4}
 Note that (1) implies (6) does not require that $\mathcal{M}(\mathcal{R}) = H^\infty (\Omega)$ but for the implication (6) implies (1), one needs somehow to conclude that the function 
 $E(\omega)$ defines a multiplier on $\mathcal{R}\otimes \mathbb{C}^n$.
\end{rem}

One can obtain another geometrical interpretation of this result by restating Proposition 4.1 on the existence of the idempotent $E$ in terms of a bounded, holomorphic module 
idempotent with range $M_\Phi$ .

\begin{pro}\label{4.5}
 Let $\{ \varphi_i \}^{n}_{i=1} \subseteq \mathcal{M}(\mathcal{R})$ for the RKHM $\mathcal{R}$ on $\Omega \subseteq \mathbb{C}^m$ satisfying (0) and let $E_\Phi$ be the holomorphic sub-bundle of
 $\Omega \times \mathbb{C}^n$ with fiber $E_\Phi (\omega)=\text{range} \Phi(\omega) \subseteq \mathbb{C}^n$.  Then (1) is equivalent to
 \begin{itemize}
  \item[(7)] There exists a complementary holomorphic sub-bundle $F$ of $\Omega \times \mathbb{c}^n$ such that 
     $\mathcal{R} \otimes \mathbb{C}^n = \{ f\in \mathcal{R} \otimes \mathbb{C}^n : f(\omega) \in E(\omega),\quad \omega \in \Omega \} \dotplus \{ f\in \mathcal{R} \otimes \mathbb{C}^n : 
     f(\omega) \in F(\omega), \quad \omega \in \Omega\}$ or
  \item[(8)] There exists a complementary submodule $\mathcal{S}$ of $\mathcal{R} \otimes \mathbb{C}^n$ such that $\mathcal{R}\otimes \mathbb{C}^n = \text{ range } M_\Phi + \mathcal{S}$.
 \end{itemize}
\end{pro}

\section{Comparing the H\"ormander and Treil-Wick Approaches}

We discuss in this section a little more about how Treil-Wick use the framework outlined in the previous section to solve the corona problem for $H^\infty (\mathbb{D})$.

The approach of Treil-Wick is related to the earlier one due to H\"ormander [5].  In this case we work on $\mathcal{R} = L^{2}_{a} (m)$, the Bergman space for the Lebesque measure $m$ on the
unit ball $\mathbb{B}^m$.  If $\{ \varphi_i \}^{n}_{i=1} \subseteq H^\infty (\mathbb{B}^m)$ satisfy (0), then one can write down an $L^\infty (m)$-solution of the corona problem as follows:\\
One defines $$\theta_j (z) = \overline{\varphi_j (z)} / \sum^{n}_{i=1} |\varphi_i (z)|^2 \text{ for } z \in \mathbb{B}^m, \text{ and }j=1,2,...,n$$ so that 
$$\sum^{n}_{i=1} \varphi_i (z)\theta_i (z) = 1 \text{ for } z \in \mathbb{B}^n.$$  But as in the Nikolski lemma, this solution is unlikely to be an $n$-tuple of holomorphic functions.

In H\"ormander's approach, $C^2$-functions $\{\alpha_j\}^{n}_{j=1}$ on $\mathbb{B}^m$ are sought so that 
\begin{itemize}
 \item [(a)] $$\sum^{n}_{i=1} \varphi_i(z)\alpha_i (z)=0\text{ for }z\in \mathbb{B}^n$$ and
 \item [(b)] $\{ \psi_i \}^{n}_{i=1} \subseteq H^\infty (\mathbb{B}^n)$, where $\psi_i=\alpha_i+\theta_i$ for $i=1,...,n$.
\end{itemize}

\noindent This amounts to solving what is referred to as a ``$\bar{\partial}$-problem with $L^\infty$-bounds''.

Let us explore a little more the relationship between the two approaches.  Consider the operator $N_A (z) : \mathbb{C}^n \to \mathbb{C}$ defined for $z\in \mathbb{B}^n$ such that 
$$N_A (z)\boldsymbol{c} = \sum^{n}_{i=1} \sum^{n}_{i=1} \alpha_i (z)C_i \text{ for } \boldsymbol{c}=(c_1, ..., c_n) \in \mathbb{C}^n.$$  Next by defining 
$$N_A: L^2 (m)\otimes\mathbb{C}^n \to L^2 (m) \text{ such that }N_A \left( \sum^{n}_{i=1} f_i \otimes e_i \right)=\sum^{n}_{i=1} \alpha_i f_i \text{ for } 
\sum^{n}_{i=1}f_i \otimes e_i \in L^2 (m)^{\otimes} \mathbb{C}^n$$ and setting $\tilde{E}=M_\Phi (N_A +N_\Theta )$, we obtain a module idempotent $\tilde{E}$ on $L^2(m) \otimes \mathbb{C}^n$ 
such that $\tilde{E}(z)$ is an idempotent onto range $\Phi(z) = \text{ range }P(z)$ for $z\in \mathbb{B}^m$.  Thus range $\tilde{E}(z)=\text{range }P(z)$ for $z \in \mathbb{B}^m$ and thus one 
is seeking a modification $A(z)$ of $\Theta(z)$ so that $N_A + N_\Theta$ is bounded and $A(z)+\Theta(z)$ is holomorphic which are essentially the same conditions as is required for $P+V$ in the
Treil-Wick approach.  By (b) it follows that $(N_A + N_\Theta )\mathcal{R} \subseteq\mathcal{R}$; thus, $\tilde{\tilde{E}}=\tilde{E}|_{\mathcal{R}\otimes\mathbb{C}^n}$ defines an idempotent on 
$\mathcal{R} \otimes \mathbb{C}^n$.  

In particular, by (b) we have $M_\Phi N_A =0$ and hence $\tilde{\tilde{E}}=M_\Phi (N_\Theta +N_A)$ is a module idempotent on $\mathcal{R} \otimes \mathbb{C}^n$.  Moreover, clearly range 
$\tilde{\tilde{E}}$ is contained in range $M_\Phi =\tilde{E} \tilde{\tilde{E}}(\mathcal{R}\otimes \mathbb{C}^n)$.  Since $\tilde{\tilde{E}}M_\Phi=M_\Phi (N_A + N_\Theta) M_\Phi = M_\Phi$, we 
see that range $\tilde{\tilde{E}}=\tilde{E}_{\mathcal{R}\otimes \mathbb{C}^n}=\text{ range }M_\Phi$ which shows that (a) and (b) yield the module idempotent $\tilde{\tilde{E}}$ on 
$\mathcal{R}\otimes\mathbb{C}^n$ which establishes (1).

In other words, in both approaches, one has a module idempotent on $L^2 (\mu) \otimes \mathbb{C}^n$ with the appropriate range function and one seeks to modify it to another idempotent
keeping the range function the same but making it bounded and holomorphic.

Both H\"ormander and Treil-Wick complete the proof for the case $\mathbb{B}^m = \mathbb{D}$ or $m=1$, but we will not explore the methods used to take the nontrivial final steps since they are 
from harmonic analysis and involve ideas using the $\bar{\partial}$ problem, Koszul complex and Hankel forms.

\section{Toeplitz Corona Problem and the CLT Property}

Another way to relate the corona problem to operator theory is by weakening the conditions on the solution functions to lie in $\mathcal{R}$ as opposed to $\mathcal{M}(\mathcal{R})$.

A multiplier $\psi \in \mathcal{M}(\mathcal{R})$ for an RKHM $\mathcal{R}$ over $\Omega \subseteq \mathbb{C}^m$ defines the Toeplitz operator 
$T^{\mathcal{R}}_{\psi} \in \mathcal{L}(\mathcal{R})$ such that $T^{\mathcal{R}}_{\psi} f=\psi f$ for $f \in \mathcal{R}$.

Let $\{ \varphi_i \}^{n}_{i=1} \subseteq \mathcal{M}(\mathcal{R})$ for some RKHM $\mathcal{R}$ over $\Omega \subseteq \mathbb{C}^m$.  One says that there is a \emph{weak solution} to the corona 
problem if
\begin{itemize}
\item [(9)] There exists an $n$-tuple $\{f_i\}^{n}_{i=1} \subseteq \mathcal{R}$ such that $$\sum^{n}_{i=1} \varphi_i (\omega) f_i (\omega)=1 \text{ for } \omega \in \Omega$$ and a solution to 
the \emph{Toeplitz corona problem} if
\item [(10)] For every $f\in \mathcal{R}$ there exists $\{f_i\}^{n}_{i=1} \subseteq \mathcal{R}$ such that $$\sum^{n}_{i=1} \varphi_i (\omega) f_i (\omega)=f(\omega) \text{ for } 
\omega \in \Omega.$$
\end{itemize}

Since this statement involves $\mathcal{R}$, one sometimes speaks of the {$\mathcal{R}$\emph{-Toeplitz corona problem}.  

There are various relationships between statement (0)-(10).  We explore some of them emphasizing a somewhat weaker version of the corona problem introduced by several authors (cf. [2]) in the 
seventies related to classical Toeplitz operators.

\begin{pro}\label{6.1}
 For $\{ \varphi_i\}^{n}_{i=1} \subseteq \mathcal{M}(\mathcal{R})$ for an RKHM $\mathcal{R}$ over $\Omega \subseteq \mathbb{C}^m$, (1) implies (10) which in turn implies (9).  Moreover, (10) is
 equivalent to either
 \begin{itemize}
  \item [(11)] $N_\Phi$ is onto or
  \item [(12)] $$\sum^{n}_{i=1} ||T^{\mathcal{R}^{*}}_{\varphi_i} f||^2 \geq \epsilon^2 ||f||^2 \text{ for }f\in \mathcal{R} \text{ for some }\epsilon>0.$$
 \end{itemize}
\end{pro}

\begin{proof}
 The equivalence of (10) and $N_\Phi$ being onto follows from the definition of $N_\Phi$.  Similarly, since 
 $$N^{*}_{\Phi} f=\sum^{n}_{i=1} T^{\mathcal{R}^*}_{\varphi_i} f \otimes e_i \text{ for } f\in \mathcal{R},$$
 $N_\Phi$ being onto is equivalent to $N^{*}_{\Phi}$ being bounded below which is statement (12).
\end{proof}

Note that the $\epsilon$ in both (0) and (12) are the same.

The question of whether (11) or (12) implies (1) is what is usually referred to as the \emph{Toeplitz corona problem} for the algebra $\mathcal{M}(\mathcal{R})$ for some RKHM $\mathcal{R}$ for 
a domain $\Omega \subseteq \mathbb{C}^m$.  Schubert in [12] showed that (12) implies (1) for $\mathcal{R}=H^2 (\mathbb{D})$ and hence (9)-(12) are all equivalent in this case.  Thus one can 
``divide'' the corona problem for $\mathcal{M}(\mathcal{R})$ into two parts: For $\{\varphi_i \}^{n}_{i=1} \subseteq \mathcal{M}(\mathcal{R})$ show that (11) holds for some
RKHM $\mathcal{R}$ with $\{ \varphi_i \}^{n}_{i=1} \subseteq \mathcal{M}(\mathcal{R})$ and second, show that (11) for $\mathcal{R}$ implies (1).

Note that although the strategy doesn't involve $\mathcal{R}$ at the conclusion, only $\mathcal{M}(\mathcal{R})$ if (1) holds, then (11) must hold for any RKHM $\mathcal{R}^{'}$ with 
$\mathcal{M}(\mathcal{R}^{'})=\mathcal{M}(\mathcal{R})$ with the same $\epsilon$.

For the special class of RKHM, which satisfy the commutant lifting theorem or have the (CLT) property, one has (11) implies (1).  Note that since $H^2 (\mathbb{D})$ satisfies the CLT, this 
shows that (11) implies (1) in this case.

\begin{defn}\label{6.2}
 \emph{A RKHM $\mathcal{R}$ is said to have the CLT property if for submodules $\mathcal{S}_1$ and $\mathcal{S}_2$ of $\mathcal{R} \otimes \mathbb{C}^m$ and $\mathcal{R}\otimes \mathbb{C}^n$,
 respectively, and a module map $X: \mathcal{R}\otimes \mathbb{C}^m / \mathcal{S}_1 \to \mathcal{R}\otimes \mathbb{C}^n / \mathcal{S}^2$, there must exist a module map 
 $\bar{X}: \mathcal{R} \otimes \mathbb{C}^m \to \mathcal{R} \otimes \mathbb{C}^n$ such that $X\pi_{S_1}=\pi_{S_2} \bar{X}$ and $||\bar{X}||=||X||$.}
\end{defn}

\begin{thm}\label{6.3}
 Let $\mathcal{R}$ be an RKHM over $\Omega \subseteq \mathbb{C}^m$ having the CLT property and $\{ \varphi_i\}^{n}_{i=1} \subseteq \mathcal{M}(\mathcal{R})$ satisfying (11).  Then (1) holds.
\end{thm}

\begin{proof}
 Set $\mathcal{S}=\ker N_{\Phi} \subseteq\mathcal{R} \otimes \mathbb{C}^n$ and then consider $Y: (\mathcal{R} \otimes \mathbb{C}^n /\mathcal{S}) \to \mathcal{R}$ such that $N_\Phi = Y \pi_S$.
 Then (11) implies that $X=Y^{-1}$ is a bounded module map and the CLT property yields $\bar{X}: \mathcal{R}\to \mathcal{R} \otimes \mathbb{C}^n$ so that $X=\pi_S \bar{X}$.  Since 
 $\pi_S=XN_\Phi$, one has $X=\pi_S \bar{X}=XN_\Phi \bar{X}$, which implies $N_\Phi \bar{X}=I_\mathcal{R}$ since $X$ is invertible.  By Lemma 3.4, $\bar{X}=\mathcal{M}_\Psi$ for some 
 $\{ \psi_i \}^{n}_{i=1} \subseteq \mathcal{M}(\mathcal{R})$ or $$\sum^{n}_{i=1} \varphi_i (\omega) \psi_i (\omega)=1 \text{ for } \omega \in \Omega ,$$ which completes the proof.
\end{proof}

\section{Toeplitz Corona Problem and Logmodular Algebras}

There is another method of showing that (11) implies (1) for some RKHM $\mathcal{R}$ which was developed by a number of authors (c.f. [2]) culminating in [11] by Trent-Wick, and then extended 
modestly in [4].  The proof given in the latter paper can be used without change to prove the following result.

\begin{thm}\label{7.1}
 Let $\mathcal{R}$ be a subnormal RKHM over $\Omega \subseteq \mathbb{C}^m$ for the probability measure $\mu$ on $\bar{\Omega}$ such that 
 \begin{itemize}
  \item [(i)] $\{k_\omega\}_{\omega \in \Omega} \subseteq \mathcal{M}(\mathcal{R}),$
  \item [(ii)] range $T^{\mathcal{R}}_{k_\omega}$ is closed for $\omega \in \Omega$; and
  \item [(iii)] for $\{ \omega_j \}^{N}_{j=1} \subseteq \Omega\text{ and } \{\lambda_j \}^{N}_{j=1} \subset \mathbb{C}^n,$ there exists $g \in \mathcal{R}$ such that 
  $$|g(z)|^2 = \sum^{N}_{j=1} |\lambda_j |^2 |k_{\omega_j} (\omega)|^2, \quad \mu \text{ a.e.}$$
 \end{itemize}
Then for $\{\varphi_i \}^{n}_{i=1} \subseteq \mathcal{M}(\mathcal{R})$ satisfying (0), (1) holds if one has
\begin{itemize}
 \item[(13)] an affirmative answer to the Toeplitz corona problem (10) for all submodules of $\mathcal{R}$ with the same $\epsilon$,
\end{itemize}
implies an affirmative answer for the corona problem (1).
\end{thm}

The idea here is to consider all the solutions in $\mathcal{R}$ and show using the von Neumann min-max theorem that one can bound the values at all finite sets of points in $\Omega$.  A normal
family argument completes the proof.  However, the Toeplitz corona solution for the family of submodules must have the same $\epsilon$ as is expressed in (13).

In the proof one considers submodules $S$ of $\mathcal{R}$ that are the range of Toeplitz operators defined by multipliers with closed range which is what conditions (i) and (ii) 
provide.  Condition (iii) allows one to replace a kind of ``convex combination'' of such submodules by one submodule.

As indicated above, the proof of the theorem and the earlier versions of it along these lines rely on an affirmative answer to the Toeplitz corona problem not just for a subnormal RKHM 
$\mathcal{R}$ on $\Omega \subseteq \mathbb{C}^m$ for a probability measure $\mu$ supported on $\bar{\Omega}$ but for a whole family of RKHM.  The family can be restricted to cyclic submodules 
where the generator can be taken to be invertible in $L^\infty (\mu)$.  In some cases it suffices to assume (11) for just $\mathcal{R}$ which is what happens in the classical case of 
$\mathcal{R}=H^2 (\mathbb{D})$, since by Beurling's theorem all (cyclic) submodules of $H^2 (\mathbb{D})$ are isometrically isomorphic to $H^2 (\mathbb{D})$ itself.  That doesn't happen very 
often (c.f. [3]) but it does if one has a subalgebra of $\mathcal{M}(\mathcal{R})$ which is a Dirichlet algebra.  But one can get by with less.

\begin{defn}\label{7.2}
 A subnormal RKHM $\mathcal{R}$ over $\Omega \subseteq \mathbb{C}^m$ for the probality measure $\mu$ supported on $\bar{\Omega}$ is said to be \emph{invertibly approximating in modulus} if for 
 every nonnegative function $\eta \in L^\infty (\mu)^{-1}$ and $\delta>0$ there exists $\psi \in \mathcal{M}(\mathcal{R})$ such that $|\eta(z)-|\psi(z)||<\delta \quad \mu \text{ a.e.}$
\end{defn}

A logmodular algebra is invertibly approximating in modulus but the converse is unclear.  In any case, the former notion has the following implication for the cyclic submodules of the RKHM 
$\mathcal{S}$ defined by a Toeplitz operator with closed range.

\begin{defn}\label{7.3}
 Two Hilbert modules $\mathcal{M}_1$ and $\mathcal{M}_2$ are said to be \emph{almost isometric} if for every $\delta>0$ there exists a module isomorphism $X:\mathcal{M}_1 \to \mathcal{M}_2$ such
 that max $(||X||,||X^{-1}||)<1+\delta$.
\end{defn}

Our interest in these two notions lies in the following results.

\begin{pro}\label{7.4}
 Let $\mathcal{R}$ and $\mathcal{R}^{'}$ be almost isometric RKHM over $\Omega \subseteq \mathbb{C}^m$.  Then
 \begin{itemize}
  \item [(i)] $\mathcal{M}(\mathcal{R})=\mathcal{M}(\mathcal{R}^{'}).$
  \item [(ii)] Inequality (12) holds for $\mathcal{R}$ iff (12) holds for $\mathcal{R}^{'}$ with the same $\epsilon$.
 \end{itemize}
\end{pro}

\begin{proof}
 Assume (12) holds for $\mathcal{R}$ for some $\delta>0$, and let $X:\mathcal{R}^{'} \to \mathcal{R}$ be a module map such that max $(||X||,||X^{-1}||)\leq 1+\delta$.  Then for 
 $\{ \varphi_i \}^{n}_{i=1} \subseteq \mathcal{M}(\mathcal{R})$ we have $T^{\mathcal{R}}_{\varphi_i}X=XT^{\mathcal{R}^{'}}_{\varphi_i}$ for $i=1,...,n$.  If (12) holds for $\mathcal{R}$ and 
 some $\epsilon >0$, then
 \begin{align*}
 &\sum^{n}_{i=1}||T^{\mathcal{R}^{'*}}_{\varphi_i} f||^2 = \sum^{n}_{i=1} ||X^{*} T^{\mathcal{R}^{'*}}_{\varphi_i} X^{*-1}f||^2 
 \geq \frac{1}{||X^{-1}||^2} \sum^{n}_{i=1}||T^{\mathcal{R}^{'*}}_{\varphi_i} X^{*-1} f||^2 \geq \\
 &\frac{\epsilon^2}{||X^{-1}||^2} ||X^{*-1}f||^2 \geq \frac{\epsilon^2}{||X^{-1}||^2 ||X||^2} ||f||^2 \geq \frac{\epsilon^2}{(1+\delta)^4} ||f||^2 .
 \end{align*}

 Hence (12) holds for $\mathcal{R}^{'}$ for $\frac{\epsilon^2}{(1+\delta)^4}$ but since $\delta>0$ is arbitrary, we have (12) for $\mathcal{R}^{'}$ with the same $\epsilon >0$.
\end{proof}

\begin{pro}\label{7.5}
 Let $\mathcal{R}$ be the subnormal RKHM $L^{2}_{a} (\mu)$ over $\Omega \subseteq \mathbb{C}^m$ and the probability measure $\mu$ supported on $\bar{\Omega}$ such that  
 $\mathcal{M}(\mathcal{R})=H^\infty (\Omega)$.  If $H^\infty (\Omega)$ is invertibly approximating in modulus, then for $\psi \in H^\infty (\Omega)$ such that $T^{\mathcal{R}}_{\psi}$ has 
 closed range $\mathcal{R}^{'}$, $T^{\mathcal{R}^{'}}_{\psi}$ has closed range and $T^{\mathcal{R}}_{\psi}$ and $T^{\mathcal{R}^{'}}_{\psi}$ are almost isometric.
\end{pro}

\begin{proof}
 First, the map $X f=f$ from the submodule range $T^{\mathcal{R}}_{\psi}$ in $L^{2}_{a} (\mu)$ to $L^{2}_{a} (|\psi|^2)$. Because range $T^{\mathcal{R}}_{\psi}$ is closed, one has
 $|\psi|$ bounded below by $\eta>0$, $\mu$ a.e.  Since $\mathcal{M}(\mathcal{R})$ is invertibly approximating in modulus, for $\delta>0$, there exists an invertible $\theta$ in 
 $\mathcal{M}(\mathcal{R})$ such that $|\theta(\omega)-|\psi|(\omega)|<\delta$ $\mu$ a.e.  Then the identity map $X$ on $L^2 (\mu)$ defines an isometrical isomorphism $X_0$ from the closure of 
 $H^\infty (\Omega)$ in $L^2 (|\psi|^2 \mu)$ to the closure of $H^\infty (\Omega)$ in $L^2 (|\theta|^2 \mu)$ such that 
 $$\text{max} \{||X_0||,||X^{-1}_{0}||\} \leq \sup_{\omega \in \Omega}
 \left\{ \text{max} \left\{\ \frac{|\theta(\omega)|}{|\psi(\omega)|}, \frac{|\psi(\omega)|}{|\theta(\omega)|} \right\} \right\}\mu \text{ a.e. }.$$
 Hence the two submodules are almost isometric or the ranges of $T^{\mathcal{R}}_{\psi}$ and $T^{\mathcal{R}}_{\theta}$ in $\mathcal{R}$ are almost isometric.  But $\theta$ invertible means the
 latter submodule is $\mathcal{R}$ which completes the proof.
\end{proof}

\begin{thm}\label{7.6}
 Let $\mathcal{R}$ be a subnormal RKHM over $\Omega \subseteq \mathbb{C}^m$ for the probability measure $\mu$ supported on $\bar{\Omega}$ such that $\mathcal{M}(\mathcal{R})=H^\infty (\Omega)$ 
 and such that $H^\infty (\Omega)$ is invertibly approximating in modulus.  Then (6) for $\mathcal{R}$ implies it for all cyclic submodules that are the range of $T^{\mathcal{R}}_{\psi}$ for 
 some $\psi \in H^\infty (\Omega)$ that has closed range.
\end{thm}

\begin{proof} 
 By hypothesis, the multiplier algebra is invertibly approximating in modulus and the rest follows by combining the previous two propositions.
\end{proof}

>From this result it follows for a subnormal RKHM $\mathcal{R}$ such that $H^\infty (\Omega)$ has the invertibly approximating in modulus property and has the Toeplitz corona property (12), 
then the corona property (1) is valid.

\begin{rem}\label{7.7}
 Unfortunately, $H^\infty (\mathbb{B}^m)$ is NOT invertibly approximating in modulus.  Although there is a function $\psi \in H^\infty (\mathbb{B}^m)$ such that 
 $|\psi (\zeta)|^2 =1+|\zeta_1 |^2$ on $\partial\mathbb{B}^m$ a.e. (cf[11]), for $m>1$ $\psi$ cannot be chosen to be invertible.  In particular, the function $z_1 \to \psi(z_1,...,z_m)$ would 
 be outer for each $(z_2,...,z_m ) \in \mathbb{B}^{m-1}$ and since it has constant modulus it must be constant.  Further consideration shows that $|\psi(z)|^2 =1+|z|^2$ for $z\in \mathbb{B}^m$ 
 which is not possible.
\end{rem}

It seems possible that results in [3] can be used to show that the assumption that $H^\infty (\Omega)$ is invertibly approximating in modulus implies that $m=1$ and that $\Omega$ is conformally
equivalent to $\mathbb{D}$.

\begin{thm}\label{7.8}
 Suppose $\mathcal{R}$ is a subnormal RKHM over $\Omega \subseteq \mathbb{C}^m$ for the probability measure $\mu$ supported on $\bar{\Omega}$ such that $H^\infty (\Omega)$ is an invertibly 
 approximately in modulus algebra.  Then (6) for $\mathcal{R}$ implies it for all cyclic submodules of $\mathcal{R}$ for the same $\epsilon$.  In particular, if the Toeplitz corona property (6)
 holds for $\mathcal{R}$, then it holds for all submodules of $\mathcal{R}$.
\end{thm}

\begin{rem}\label{7.9}
 It seems unlikely that all pairs of almost isometric Hilbert modules are in fact unitarily equivalent.  However, that is the case for a Hilbert module $\mathcal{M}$ over 
 $\mathbb{C}[z_1,...,z_n]$ for which the common eigenspaces are finite dimensional and they span $\mathcal{M}$.
\end{rem}

\begin{rem}\label{7.10}
 In [5] Hamilton-Raghupathi show that one can use another development in operator theory, factorization in dual algebras, to prove some general Toeplitz corona theorems.  Combining these 
 results with some results of Prunaru yields Toeplitz corona theorems for Bergman spaces over certain domains in $\mathbb{C}^m$.
\end{rem}

\section{Taylor Spectrum and the Corona Problem}
In [7] Taylor introduced a notion of joint spectrum for $n$-tuples of commuting elements in a Banach algebra which also applies to an $n$-tuple of commuting operators $(T_1, ..., T_n)$ on a
Hilbert space $\mathcal{H}$.  The Taylor spectrum, $\sigma^{TAY} (T_1,...,T_n)$, is a nonempty compact subset of $\mathbb{C}^n$ for which there is a good functional calculus.  Moreover, Taylor 
shows that the existence of an $n$-tuple $(S_1,...,S_n)$ of operators on $\mathcal{H}$ that commute with each other and the $\{T_i\}^{n}_{i=1}$ implies that 
$\boldsymbol{0}=(0,...,0) \notin \sigma^{TAY}(T_1,...,T_n)$.  In particular, for $\{\varphi_i\}^{n}_{i=1} \subset \mathcal{M}(\mathcal{R})$ for a RKHM $\mathcal{R}$ over 
$\Omega \subseteq \mathbb{C}^n$, a necessary condition for (1) to hold or for the corona problem to have an affirmative solution is for
\begin{itemize}
 \item [(13)] $\boldsymbol{0} \notin \sigma^{TAY} (T^{\mathcal{R}}_{\varphi_1}, ..., T^{\mathcal{R}}_{\varphi_n})$.
\end{itemize}

The origin $\boldsymbol{0}$ is not in the Taylor spectrum precisely when the Koszul complex, built from the $n$-tuple $(T^{\mathcal{R}}_{\varphi_1},...,T^{\mathcal{R}}_{\varphi_n})$, is exact.
We won't recall the definition of the full Korzul complex but only for the sequence at the first and the last nonzero modules which are 
$0 \to \mathcal{R} \xrightarrow{M_\Phi} \mathcal{R} \otimes \mathbb{C}^n$ and $\mathcal{R} \otimes \mathbb{C}^n \xrightarrow{N_\Phi} \mathcal{R} \to 0$.  Hence, \emph{(13) or exactness of the 
Koszul complex} implies that the range of $M_\Phi$ is closed and that $N_\Phi$ is onto.  In particular, the assumption of exactness implies (10) or that the Toeplitz corona problem has an 
affirmative solution.  These considerations suggest the following:

\begin{que}\label{8.1}
Does the exactness of the Koszul complex for $(T^{\mathcal{R}}_{\varphi_1},...,T^{\mathcal{R}}_{\varphi_n})$ for some RKHM $\mathcal{R}$ over $\Omega \subseteq \mathbb{C}^n$ imply (1) or that 
the corona problem has an affirmative solution?
\end{que}

For that to be true, one would know that for two RKHM
$\mathcal{R},\mathcal{R}^{'}$ with $\mathcal{M}(\mathcal{R}) = \mathcal{M}(\mathcal{R}^{'}),\\
 \sigma^{TAY}(T^{\mathcal{R}}_{\varphi_1},...,T^{\mathcal{R}}_{\varphi_n})=\sigma^{TAY}(T^{\mathcal{R}^{'}}_{\varphi_1},...,T^{\mathcal{R}^{'}}_{\varphi_n})$.

In [8] Taylor considers in more detail the possible implication of when $\boldsymbol{0} \notin \sigma^{TAY}(T^{\mathcal{R}}_{\varphi_1},...,T^{\mathcal{R}}_{\varphi_n})$, implies (1).  We 
prove the following related result after introducing some notions and propositions.

\begin{thm}\label{8.2}
 Let $\{ \varphi_i\}^{n}_{i=1} \subseteq \mathcal{M}(\mathcal{R})$ for the RKHM $\mathcal{R}$ over $\Omega \subseteq \mathbb{C}^m$ such that $\boldsymbol{0} \notin$ closed 
 polynomial convex hull of $\{ \Phi (\omega): \omega \in \Omega\} \subseteq \mathbb{C}^n.$  Then (1) holds or the corona problem has an affirmative solution.
\end{thm}

For a bounded subset $X\subseteq \mathbb{C}^n$, its polynomial convex hull, $\text{Pol}(X)$, is defined
$$\text{Pol}(X)=\{z_0 \in \mathbb{C}^n : |p(z)| \leq \sup_{z \in X} |p(z)| \text{ for } p\in \mathbb{C} [z_1, ..., z_n] \}.$$

Let $P(X)$ denote the function algebra obtained by completing the restriction of $\mathbb{C}[z_1, ..., z_n]$ to $X$ in the supremum norm.  (Note that if $\bar{X}$ is the closure of $X$, then
$\text{Pol}(\bar{X})=\text{Pol}(X)$ and $P(\bar{X})=P(X)$.)  One way in which the polynomial convex hull arises is in the following result:

\begin{pro}\label{8.3}
 For a bounded subset $X$ of $\mathbb{C}^n , \mathcal{M}_{P(X)} = \text{Pol}(X).$
\end{pro}

\begin{proof}
 First, note that for $z_0 \in \text{Pol}(X)$, it follows that evaluation at $z_0$ defines a bounded multiplicative linear functional on $P(X)$ which yields an embedding of $\text{Pol}(X)$ into 
 $M_{P(X)}$.  Second, if $L$ is a bounded multiplicative linear functional on $P(X)$, then $(L(z_1),...,L(z_n)) \in \mathbb{C}^n$ determines a point in $\text{Pol}(X)$ at which 
 evaluation is identical to $L$.  This identification completes the proof.
\end{proof}

\begin{proof} \emph{Theorem 8.2:} If $\Phi=(\varphi_1, ..., \varphi_n): \Omega \to \mathbb{C}^n$, then the assumption that $\boldsymbol{0} \notin \text{Pol}(\Phi(\Omega))$ implies the 
existence of functions $\{ \sigma_i \}^{n}_{i=1} \subseteq P(\Phi(\Omega))$ satisfying $$\sum^{n}_{i=1} z_i \sigma_i (z_1, ..., z_n)=1 \text{ for } (z_1, ..., z_n) \in \text{Pol}(\Phi(\Omega)).$$
This follows because the ideal generated by $z_1,...,z_n$ in $\text{Pol}(X)$ can't be proper or evaluation at $\boldsymbol(0)$ would define a multiplicative linear functional on
$P(X)$.  Hence, the ideal contains the constant function $1$.  Now setting 
$\psi_i(\omega_1, ..., \omega_m),=\sigma_i (\varphi_i (\omega_1, ..., \omega_m), ..., \varphi_n (\omega_1,...,\omega_m)) \text{ for }(\omega_1,...,\omega_m)\in \Omega$ yields functions 
$\{ \psi_i \}^{n}_{i=1} \subseteq H^{\infty} (\Omega)$ satisfying $$\sum^{n}_{i=1} \varphi_i(\omega) \psi_i(\omega)=1 \text{ for }\omega \in \Omega,$$ which completes the proof.
\end{proof}

\begin{rem}\label{8.4}
 If the basic algebra is not $H^\infty (\Omega)$ but is some other Banach algebra of holomorphic functions $\mathcal{A}$, then it seems likely that one could adapt the argument in Taylor [9] to 
 reach the same conclusion with the functions constructed being in $\mathcal{A}$.
\end{rem}

\begin{cor}\label{8.5}
 For $\{ \varphi_i\}^{n}_{i=1} \subseteq \mathcal{M}(\mathcal{R})=H^\infty (\Omega)$ for the RKHM $\mathcal{R}$ over $\Omega \subseteq \mathbb{C}^m$, if there is a 
 $p(z) \in \mathbb{C}[z_1,...,z_n]$ such that $$|p(\boldsymbol{0})| > \sup_{\omega \in \Omega} |p(\varphi_1 (\omega)) |p(\varphi_1 (\omega),...,\varphi_n (\omega)),$$ then (1) holds or the 
 corona problem has an affirmative solution.
\end{cor}

As indicated in the previous remark, it is likely that one doesn't need to assume $\mathcal{M}(\mathcal{R})=H^\infty (\Omega)$.

\vspace{5mm}

\noindent Ronald G. Douglas\\                             
Texas A\&M University\\
College-Station, TX 77843-3368\\
rgd@tamu.edu


\begin{thebibliography}{8}
\bibitem{1}
Lennart Carleson, Interpolation by bounded analytic functions and the Corona problem, Ann. Math. (2) 76 (1962), 547-559.

\bibitem{2}
Ronald G. Douglas, Steve Krantz and Brett Wick, A History of the Corona Problem, pre-print.

\bibitem{3}
Ronald G. Douglas and Jaydeb Sarkar, On unitary equivalent submodules, Indiana U. Math. J. 57 (2008), no. 6, 2729-2743.

\bibitem{4}
Ronald G. Douglas and Jaydeb Sarkar, Some remarks on the Toeplitz corona problem, Hilbert spaces of analytic functions, CRM Proc. Lecture Notes 51, Amer. Math. Soc. Providence, RI (2010), 
81-89.

\bibitem{5}
Ryan Hamilyon and Mrinal Raghupathi, The Toeplitz corona problem for multipliers on Nevanlinna-Pick space, Indiana J. Math. (to appear).

\bibitem{6}
Lars H\"ormander, \emph{An introduction to complex analysis in several variables}, van Nostrand, Princeton, 1966.

\bibitem{7}
C. Jiang and Z. Wang, \emph{Structure of Hilbert Space Operators}, World Scientific, Singapore (2006).

\bibitem{8}
I. J. Schark, Maximal ideals in an algebra of bounded analytic functions, J. Math. Mech. 10 (1961), 735-746.

\bibitem{9}
C.F. Schubert, The corona theorem as an operator theorem, Proc. Amer. Math. Soc. 69 (1978), no. 1, 73-76.

\bibitem{10}
Joseph L. Taylor, A joint spectrum for several commuting operators, J. Funct. Annl.

\bibitem{11}
Joseph L. Taylor, The Analytic Functional Calculus for Several Commuting Operators, Acta Math. 125 (1970), 1-38.

\bibitem{12}
Sergei Treil and Brett D. Wick, Analytic projections, corona problem and geometry of holomorphic vector bundles, J. Amer. Math. Soc. 22 (2009), no. 1, 55-76.

\bibitem{13}
Tavan T. Trent and Brett D. Wick, Toeplitz corona theorems for the polydisk and the unit ball, Complex Anal. Oper. Theory 3 (2009), no. 3, 729-738.
\end{thebibliography}
\end{document}